\documentclass[12pt]{amsart}
\usepackage{graphics,wrapfig, graphicx, mathrsfs,xcolor}
\usepackage{mathtools}
\usepackage{tikz}
\usepackage{amssymb,amscd,graphicx,xcolor,subfig,tikz}
\usepackage{graphics}
\usepackage{indentfirst}
\usepackage{bm, enumerate,xcolor}
\usepackage[dvips]{epsfig}
\usepackage{latexsym}
\usepackage{float}

\usepackage{amsmath}
\usepackage{amssymb}
\usepackage[applemac]{inputenc}
\usepackage[T1]{fontenc}
\newtheorem{lemma}{Lemma}[section]
\newtheorem{theorem}{Theorem}[section]

\newtheorem{proposition}{Proposition}[section]

\newtheorem{remark}{Remark}[section]

\numberwithin{equation}{section} \numberwithin{theorem}{section}
\numberwithin{example}{section} \numberwithin{remark}{section}
\numberwithin{figure}{section} \numberwithin{algorithm}{section}
\mathtoolsset{showonlyrefs}

\title[Area growth estimate]{An area growth estimate of the Liouville equation}
\author{Xiaohan Cai}
\address{School of Mathematical Sciences, Shanghai Jiao Tong University, Shanghai 200240}\email{xiaohancai@sjtu.edu.cn}
\author{Mijia Lai}
\address{School of Mathematical Sciences, Shanghai Jiao Tong University, Shanghai 200240}\email{laimijia@sjtu.edu.cn}
\author{Chilin Zhang}
\address{School of Mathematical Sciences, Fudan University, Shanghai 200433, China}\email{zhangchilin@fudan.edu.cn}

\begin{document}
\begin{abstract}
    We establish an area growth estimate for solutions that are bounded from above of the Liouville equation
    \begin{equation}
        \Delta u+K e^{2u}=0
    \end{equation}
    with a positive pinched curvature $0<\lambda\leq K\leq\Lambda$. As an application, we provide a new proof of Eremenko-Gui-Li-Xu's result in \cite{EGLX}. We also classify solutions with an upper bound in the half plane with the boundary having constant geodesic curvature.
\end{abstract}
\maketitle

\section{Introduction} 

In this paper, we are concerned with the classification of the entire solution $u$ of the Liouville equation 
\begin{equation}\label{eq. EqL}
\Delta u+e^{2u} =0
\end{equation}
in the plane. 

Historically, it was Monge who first derived this equation in the study of prescribing constant Gauss curvature problem. Geometrically, if $u$ is a solution of \eqref{eq. EqL}, then the conformal metric $g=e^{2u} g_0$ has constant Gauss curvature $K_g\equiv 1$. Liouville wrote down a representation formula for the solution in 1853 \cite{L}. More precisely, let $\Omega$ be a simply connected domain in $\mathbb{R}^2$, then all real solutions of (\ref{eq. EqL}) in $\Omega$ are of the form
    \begin{align}\label{eq.Liouville's formula}
        u(z)=\ln \frac{2|f'(z)|}{1+|f(z)|^2},
    \end{align}
where $f$ is a locally univalent meromorphic function with simple poles in $\Omega$. Identifying the pole with the north pole of the Riemann sphere, $f$ can be viewed as a map $f: \Omega\to \mathbb{S}^2$, and the pull back of the round unit sphere metric via $f$ is exactly $e^{2u}g_0$. We refer $f$ as a developing map of the solution $u$, which is unique up to a composition with $\frac{pz-\bar{q}}{qz+\bar{p}}$, with $|p|^2+|q|^2=1$.
Later, Chou-Wan \cite{CW} and Brito-Hounie-Leite \cite{BHL} extend the Liouville's formula to non simply connected domains. 

From PDE's viewpoint, it is desirable to derive explicit expressions for $f$ under various assumptions. In their seminal paper \cite{ChL}, Chen and Li proved that all solutions with finite total curvature ($\int_{\mathbb{R}^2} e^{2u} dx<\infty$) must be radially symmetric. Numerous alternative proofs of this Chen-Li result have emerged subsequently, see e.g., \cite{CW, CK, HW, JW}. We also refer the readers to a recent survey article \cite{CaL}.

Eremenko-Gui-Li-Xu \cite{EGLX} made a breakthrough on the Liouville equation recently. Among other things, they classified all solutions of \eqref{eq. EqL} that are bounded from above. 
\begin{theorem}[Eremenko-Gui-Li-Xu]\label{thm. EGLX}
   Let $u$ be an entire solution to the Liouville equation \eqref{eq. EqL}, then $u$ is bounded from above, if and only if either
$$
u(z)=\ln \left(\frac{2}{1+|z|^2}\right),
$$
or
\begin{align} \label{eq: t}
u_t(x, y)=\ln \left(\frac{2 e^x}{1+t^2+2 t e^x \cos y+e^{2 x}}\right) \quad \text { for some } \quad t \geq 0,
\end{align}
up to a normalization. By normalization, we mean that if $u(z)$ is an entire solution, then 
    \begin{align}\label{transformation}
  u_{\lambda, z_0}(z):=u(\lambda z+z_0)+\ln |\lambda|
        \quad \text{for some } z_0\in\mathbb{C},\lambda\in\mathbb{C}\setminus\{0\}
    \end{align}
is also a solution.
\end{theorem}

The former solution is the Chen-Li's finite total curvature solution, which yields the standard round unit sphere. The corresponding developing function for this solution is simply $f(z)=z$. The $t$-family solution $u_t$ emerges from the developing function $f(z)=e^z+t$. This map describes the universal cover of the unit sphere with two punctures. 

Eremenko-Gui-Li-Xu's proof relies heavily on the Nevanlinna theory and the asymptotic integration theory for the complex linear ODE of the form 
\begin{equation}
w^{\prime \prime}+P(z) w=0, \quad z \in \mathbb{C}.
\end{equation}

We would also like to mention that the Liouville equation is closely related to the minimal surfaces, see \cite{BLdSN}. In particular, Erenmenko-Gui-Li-Xu's theorem in fact indicates the following: let $\mathbb{X}: \mathbb{R}^2\to \mathbb{R}^3$ be a nonumbilic minimal immersion with bounded Gauss curvature, then $\mathbb{X}(\mathbb{R}^2)$ must be isometric to one of followings: Enneper surface, Catenoid and Bonnet family of minimal surfaces. We refer the reader to Section 5.1 of \cite{CaL} for details.

In this paper, we present a novel proof for Theorem~\ref{thm. EGLX}, one that embraces a more PDE-oriented approach. The strategy of our proof is inspired  by a proof of Chen-Li's result by Hang and Wang \cite{HW}. Here, we outline the key elements of our approach.

The first ingredient is an observation that if $u$ is a solution of \eqref{eq. EqL}, then
\begin{equation}
    u_{zz}-u_z^2
\end{equation}
is holomorphic. Under the assumption $\int_{\mathbb{R}^2} e^{2u} dx<\infty$, Hang and Wang obtained some elliptic estimates showing that $|\nabla u|=o(1)$ and $|\nabla^2 u|=o(1)$. Consequently $u_{zz}-u_z^2\equiv 0$ and the classification follows. 

We observe first that $u_{zz}-u_z^2$ being a non-zero constant corresponds to a solution $u_t$ for some $t$.  Our objective is then to establish a relatively relaxed elliptic estimate for $u$, premised on the assumption that it is bounded from above. To this end, we derive a "self-similar" Harnack inequality, detailed in Lemma~\ref{lem. self similar harnack}.  It leads to an area growth estimate for the solution of the Liouville equation which is bounded from above. 

Here is the statement of our main result: the area growth estimate for the solution of the Liouville equation with positive pinched prescribed curvature. Indeed we can treat more general equations, for example solutions with a global upper bound to the equation $\Delta u=-\chi_{\{u\geq0\}}$.
\begin{theorem}\label{thm(main). density estimate}
    Let $u$ be an entire solution of
    \begin{equation}\label{eq. EqL variable curvature}
        \Delta u+K e^{2u}=0 
    \end{equation}
    which is bounded from above. Here $K$ is a measurable function satisfying $0<\lambda\leq K\leq\Lambda$. Then there exists $\epsilon>0$ such that 
    \begin{equation}
    \int_{B_R} e^{2u} dx =O(R^{2-\epsilon}), \text{as $R \to \infty$}.
    \end{equation}
\end{theorem}
As a direct consequence, we obtain growth rate of $u$ up to second derivatives.
\begin{theorem}\label{thm(main). derivative growth rate}
    Under the assumptions of Theorem~\ref{thm(main). density estimate}, we have 
    \begin{equation}
        \|u\|_{L^{\infty}(B_{R})}=O(R^{2-\epsilon}),\quad\|Du\|_{L^{\infty}(B_{R})}=O(R^{1-\epsilon^{2}/2}).
    \end{equation}
    If we further assume $[K]_{C^{\alpha}(\mathbb{R}^{2})}<\infty$ in \eqref{eq. EqL variable curvature}, then
    \begin{equation}
        \|D^{2}u\|_{L^{\infty}(B_{R})} 
        =O(R^{\alpha}\ln{R}).
    \end{equation}
\end{theorem}

It turns out that the area growth estimate is related to the Nevanlinna's theory in a weaker form. For a meromorphic function $f$, its spherical derivative is defined as 
\begin{equation}
f^{\#}(z):=\frac{2\left|f^{\prime}(z)\right|}{1+|f(z)|^2}.
\end{equation}
Let 
\begin{equation}
A(r):=\frac{1}{4\pi} \int_{|x|\leq r} (f^{\#}(z))^2 |dz|,
\end{equation}
and $T(r)=\int_{0}^r \frac{A(t)}{t} dt$. 
The order of $f$ is defined by
$$
\rho(f):=\underset{r \rightarrow \infty}{\limsup } \frac{\log T(r)}{\log r}.
$$

Nevanlinna theory describes locally univalent meromorphic functions $f$ of finite order  as $f=\frac{w_1}{w_2}$, where $w_1,w_2$ are two linearly independent solutions of 
\begin{equation}
w''(z)+P(z)w(z)=0.
\end{equation}
Here $P:=\frac{1}{2}\mathcal{S}(f)$, the Schwartz derivative of $f$, which is defined by 
\begin{equation}
\mathcal{S}( f):=\left(\frac{f^{\prime \prime}}{f^{\prime}}\right)^{\prime}-\frac{1}{2}\left(\frac{f^{\prime \prime}}{f^{\prime}}\right)^2.
\end{equation} 
Moreover, the degree of $P$ is related to the order of $f$ by $d(P)=2(\rho(f)-1)$, if $P\neq 0$.

When $f$ is a  developing function corresponding to a solution $u$, it is readily seen that $A(r)=\frac{1}{4\pi}\int_{B_r} e^{2u} dx$, which is exactly the conformal area that are concerned with. Moreover an easy computation shows that $P(z)=u_{zz}-u_z^2$.

Therefore in view of the Theorem ~\ref{thm(main). density estimate}, 
it is easy to see that $\rho(f)\leq 2-\epsilon,$ where $f$ is any developing function corresponding to the solution of \eqref{eq. EqL} that are bounded from above. Consequently, $d(P)\leq 1$ and the classification result would follow.
It is thus also reasonable to conjecture the best constant $\epsilon$ in Theorem~\ref{thm(main). density estimate} should be $\epsilon=1$.

With Theorem~\ref{thm. EGLX} in mind and together with a half-space version of the Harnack inequality, we also classify half space solutions of the Liouville equation that are bounded from above and with constant geodesic curvature on the boundary. Previously, Zhang \cite{Z} classified finite total curvature solutions of the Liouville equation on the half space.

\begin{theorem}[Classification in the half plane]\label{thm. half plane}
    Let $\kappa\in\mathbb{R}$ be a constant and let $u:\mathbb{R}^{2}_{+}\to\mathbb{R}$ be a solution which is bounded from above to
    \begin{equation}\label{eq. half plane: boundary condition}
    \left\{
  \begin{array}{ll}
    \Delta u+e^{2u}=0, & \text{in $\mathbb{R}^2_{+}$}, \\
    \frac{\partial u}{\partial x_2}(x_1,0)=-\kappa e^{u(x_1,0)}, & \text{on $x_2=0$}.
  \end{array}
\right.
    \end{equation}
 
Then we have the following classification result up to normalizations:
 either 
    \begin{equation}
        u=\ln \left( \frac{2  }{1 +x_1^2+\left(x_2+\kappa \right)^2}\right),
       \end{equation}
        or
        \begin{equation}
 u=\ln\left( \frac{2e^{x_2}}{\sqrt{\kappa^2+t^2+1}-\kappa+ 2t \cos(x_1)e^{x_2}+(\sqrt{\kappa^2+t^2+1}+\kappa)e^{2x_2}}\right), \quad \text{for some } t\geq 0,
\end{equation}
    or
        \begin{equation}
    u=\ln \left(\frac{2  e^{ x_1}}{(1+t^2)+2te^{ x_1}\cos(x_2- b)+e^{ 2x_1}}\right)
\end{equation}
with $b$ chosen such that $t\sin{b}=\kappa$.

Here a normalization of \eqref{eq. half plane: boundary condition} is of the form:
\[
u_{\lambda, a}(x_1, x_2):=u(\lambda x_1+a, \lambda x_2)+\ln \lambda \quad (\lambda>0 , a\in \mathbb{R}).
\] 
\end{theorem}

The developing function of the second family of solutions in Theorem \ref{thm. half plane} is 
%$f(z)=a\frac{e^{iz}+t_0}{1+t_0e^{iz}}$, where $a=\sqrt{\kappa^2+1}-\kappa,\ t_0=\frac{1}{t}(\sqrt{\kappa^2+t^2+1}-\sqrt{\kappa^2+1})$.
 $f(z)=\left(\sqrt{\kappa^2+t^2+1}-\kappa\right)e^{iz}+t$.
This map describes the universal cover of the geodesic ball $B_p\left(2\arctan a\right)$ on $\mathbb{S}^2$ with a puncture $q$ satisfying $dist(p,q)=2\arctan (t_0 a)$, where $a=\sqrt{\kappa^2+1}-\kappa\in (0,+\infty), t_0=\frac{1}{t}(\sqrt{\kappa^2+t^2+1}-\sqrt{\kappa^2+1})\in(0,1)$. The developing function of the third family of solutions in Theorem \ref{thm. half plane} is $f(z)=e^{z-ib}+t$, and the corresponding manifold $(\mathbb{R}^2_+,e^{2u}g_0)$ is a cut off of the manifold $(\mathbb{R}^2,e^{2u_t}g_0)$ appeared in Theorem \ref{thm. EGLX}.

We also would like to mention that G\'alvez and Mira \cite{GM} have classified all solutions of \eqref{eq. half plane: boundary condition} in terms of the Liouville formula with a precise behavior of the developing function on the boundary.

The organization of the paper is as follows. In Section 2,  we present key technical tools to prove the area growth estimate, the Theorem~\ref{thm(main). density estimate}. In Section 3, we appeal to Schauder estimates to yield growth estimates of $u$ itself and present a new proof of Theorem~\ref{thm. EGLX}. In section 4, we discuss the classification results on the half plane.

\section{Technical tools and the area growth estimate}\label{Section. essential tools}

In this section, we present essential technical tools that are needed and then prove Theorem~\ref{thm(main). density estimate}.

For clarity, we make a list of the notations we use.
\begin{itemize}
    \item $\mathbb{R}^{2}$: the standard $2$-dimensional Euclidean space.
    \item $B_{R}(x)$: the standard Euclidean disc in $\mathbb{R}^{2}$, with radius $R$ and centered at $x\in\mathbb{R}^{2}$.
    \item $B_{R}$: a short hand for $B_{R}(0)$.
    \item $2\mathbb{Z}^{+}+1$: The set of odd integers no less than $3$, equivalent to $\{3,5,7,9,\cdots\}$.
    \item $Q_{R}(x):=x+[-\frac{R}{2},\frac{R}{2}]^{2}$: a square centered at $x$ with side length $R$.
    \item $Q_{R}$: a short hand for $=Q_{R}(0)$.
    \item $\mathcal{D}_{u}(\Omega):=\frac{1}{|\Omega|}\int_{\Omega}e^{2u}dx$:  area density of a function. 
    \item $\delta_{u}(R)=\sup_{x\in\mathbb{R}^{2}}\mathcal{D}_{u}(Q_{R}(x))$: global density of $u$ at squares of side length $R$.
\end{itemize}

\subsection{A "self-similar" Harnack principle}
Let's now state a version of Harnack principle we need in this paper. It has its own independent interest.

It is known that in $\mathbb{R}^{2}$, the fundamental solution can be given by
\begin{equation}
    \Gamma_{R}(x)=\frac{1}{2\pi}\ln{\frac{|x|}{R}},\quad\Delta\Gamma_{R}=\delta_{0}.
\end{equation}
We follow Caffarelli's explanation of the mean value property and set
\begin{equation}\label{kernel for MVP}
    \gamma_{R}(x)=\frac{|x|^{2}-R^{2}}{4\pi R^{2}}-\Gamma_{R}(x).
\end{equation}
One can verify that for $x\in B_{R/2}$, we have
\begin{equation}\label{singular integration rate}
    \gamma_{R}(x)\sim\ln{\frac{R}{|x|}}.
\end{equation}
Besides,
\begin{equation}
    \Delta\gamma_{R}=\frac{1}{\pi R^{2}}-\delta_{0}\mbox{ in }B_{R},\quad\nabla\gamma_{R}=0\mbox{ on }\partial B_{R}.
\end{equation}
Let $w$ be a function defined in $B_{R}$, then we integrate-by-parts and obtain
\begin{equation}\label{mean value property}
    \mathop{avg}\limits_{B_{R}}w-w(0)=\int_{B_{R}}w\Delta\gamma_{R}=\int_{B_{R}}\gamma_{R}\Delta w.
\end{equation}

We now develop a Harnack estimate. You can search for the definition of $Q_{N^{k}}(x)$ in the notation.
\begin{lemma}\label{lem. self similar harnack}
    Let $N\in2\mathbb{Z}^{+}+1$ and $\sigma\in(2/N^{2},1)$. Let $k\geq0$ be an arbitrary integer. Assume that $w:\mathbb{R}^{2}\to[0,\infty)$ satisfies the equation $\Delta w=f$ and $f$ satisfies the following two conditions:
    \begin{itemize}
        \item[(1)] $|f|\leq1$ everywhere in $\mathbb{R}^{2}$;
        \item[(2)] For any $x\in\mathbb{R}^{2}$ and any $i\in\{0,1,\cdots,k\}$,
        \begin{equation}
            \frac{1}{N^{2i}}\int_{Q_{N^{i}}(x)}|f|dx\leq\sigma^{i}.
        \end{equation}
    \end{itemize}
    Then there exists a universal constant $C$ (independent of $(N,\sigma,k)$), such that for any $x\in\mathbb{R}^{2}$ and $y\in Q_{N^{k}}(x)$, we have
    \begin{equation}\label{eq. Harnack formula}
        w(y)\leq C\Big(w(x)+N^{2k}\sigma^{k}\ln{N}\Big).
    \end{equation}
\end{lemma}

Before proving Lemma~\ref{lem. self similar harnack}, let's have a brief digestion. From the assumptions of $f$, we know $f\in L^{\infty}(\mathbb{R}^{2})\cap L^{1}_{loc}$. Since $L^{1}_{loc}$ space is a critical space in $\mathbb{R}^{2}$, we don't have the standard Harnack inequality, but rather will get a logarithmic term of $\|f\|_{L^{1}}$ on the right-hand-side, for example
\begin{equation}
    w(y)\leq C\Big(w(x)+k\cdot N^{2k}\sigma^{k}\ln{N}\Big).
\end{equation}
The reason why we don't see $\log\big(\|f\|_{L^{1}}\big)\sim k\ln{N}$ in \eqref{eq. Harnack formula} is that the assumption on $f$ has a certain sense of "self-similarity" similar to a Cantor set.

Having seen how Lemma~\ref{lem. self similar harnack} differs from the standard Harnack principle for equations with right-hand-sides of critical integrability, we can now work on its proof. Throughout the proof, all constants like $C_{1}-C_{4}$ are universal.
\begin{proof}[Proof of Lemma~\ref{lem. self similar harnack}]
    The idea is to use the mean value property \eqref{mean value property} and obtain the following claim.
    \begin{itemize}
        \item Claim: There exists a constant $C_{1}$, such that for any $x\in\mathbb{R}^{2}$ and any $R\in[N^{k}/\sqrt{2},\sqrt{2}N^{k}]$, we have
        \begin{equation}
            |\mathop{avg}\limits_{B_{R}(x)}w-w(x)|\leq C_{1}N^{2k}\sigma^{k}\ln{N}.
        \end{equation}
    \end{itemize}
    After that it suffices to compare $\mathop{avg}\limits_{B_{\sqrt{2}N^{k}}(x)}w$ with $\mathop{avg}\limits_{B_{N^{k}/\sqrt{2}}(y)}w$ and obtain \eqref{eq. Harnack formula}.

    Now let's prove the claim, and for simplicity assume $x=0$. By \eqref{mean value property},
    \begin{equation}
        |\mathop{avg}\limits_{B_{R}}w-w(0)|=|\int_{B_{R}}\gamma_{R}\Delta w|\leq\int_{Q_{3N^{k}}}(\gamma_{R})^{+}|f|,
    \end{equation}
    where $(\gamma_{R})^{+}:=\max\{\gamma_{R},0\}$. By \eqref{singular integration rate} we have
    \begin{equation}
        \gamma_{R}(x)\leq C_{2}(k-i+1)\ln{N}
    \end{equation}
    outside $Q_{N^{i-1}}$. By the assumptions of $f$, we have
    \begin{align}
        \int_{Q_{3N^{k}}}(\gamma_{R})^{+}|f|=&\Big\{\sum_{i=1}^{k-1}\int_{Q_{N^{i}}\setminus Q_{N^{i-1}}}+\int_{Q_{3N^{k}}\setminus Q_{N^{k-1}}}+\int_{Q_{1}}\Big\}(\gamma_{R})^{+}|f|\\
        \leq&C_{3}\Big\{\sum_{i=1}^{k}(k-i+2)N^{2i}\sigma^{i}\ln{N}+\int_{Q_{1}}(k\ln{N}+\ln{\frac{1}{|x|}})\Big\}\\
        \leq&C_{4}N^{2k}\sigma^{k}\ln{N},
    \end{align}
    where we use $N^{2}\sigma\geq2$ in the last step. This finishes the proof of Lemma~\ref{lem. self similar harnack}.
\end{proof}
\begin{remark}\label{rmk. self similar harnack domain}
    In fact, we just need to ensure that $w$ is defined inside $Q_{3N^{k}}(x)$ instead of in the whole $\mathbb{R}^{2}$.
\end{remark}
\begin{remark}
    The assumption that $N$ is an odd integer is purely given for computation convenience. In fact, when integrating inside the region $Q_{N^{i}}\setminus Q_{N^{i-1}}$, we simply divide it into $N^{2}-1$ sub-squares, each of which has side length $N^{i-1}$, and add up the integration in each sub-square.
\end{remark}
\subsection{A maximum principle}
Reversing the inequality in assumption (2) of Lemma~\ref{lem. self similar harnack}, one could derive a quantitative version of maximum principle which is useful in our subsequent contradiction argument. 
\begin{lemma}\label{lem. maximum principle}
    Let $N\in2\mathbb{Z}^{+}+1$ and $\sigma\in(2/N^{2},1)$. Let $k\geq0$ be an arbitrary integer. Assume that $w:\mathbb{R}^{2}\to[0,\infty)$ satisfies  $\Delta w\geq f$ and $f$ satisfies the following two conditions:
    \begin{itemize}
        \item[(1)] $0\leq f\leq1$ everywhere in $\mathbb{R}^{2}$;
        \item[(2)] For any $x\in\mathbb{R}^{2}$ and any $i\in\{0,1,\cdots,k\}$,
        \begin{equation}
            \frac{1}{N^{2i}}\int_{Q_{N^{i}}(x)}f dx\geq \frac{\sigma^{i}}{2}.
        \end{equation}
    \end{itemize}
    Then there exists a universal constant $C$ (independent of $(N,\sigma,k)$), such that for any $x\in\mathbb{R}^{2}$, we have
    \begin{equation}\label{eq. maximum principle formula}
        w(x)\leq  \max_{\partial Q_{N^{k+1}}(x)}w-
        C\sigma^k N^{2k+2}.
    \end{equation}
\end{lemma}
\begin{proof}
    The proof follows the same strategy as that of Lemma~\ref{lem. self similar harnack} with a different stratification of $Q_{N^{k+1}(x)}.$

    Without loss of generality, we assume $x=0$. Green's formula implies that
    \begin{align*}
        w(0)&=\int_{Q_{N^{k+1}}} G_0(y)\Delta w(y)dy
        +\int_{\partial Q_{N^{k+1}}}\frac{\partial G_0}{\partial n}w(y)d S_y\\
        &\leq \int_{Q_{N^{k+1}}} G_0(y)f(y)dy+\max_{\partial Q_{N^{k+1}}} w,
    \end{align*}
    where $G_0$ is the Green function centered at $0$ and supported in $Q_{N^{k+1}}$. 

    By definition, $G_0(y)=\frac{1}{2\pi}\ln |y|+h(y)$, where $h$ is a harmonic function on $Q_{N^{k+1}}$ with boundary value $h(y)=-\frac{1}{2\pi}\ln |y|$ on $\partial Q_{N^{k+1}}$. The maximum principle implies that
    \begin{align*}
        G_0(y)\leq \frac{1}{2\pi}\ln \frac{2|y|}{N^{k+1}},\quad \forall y\in Q_{N^{k+1}}.
    \end{align*}
    It follows that 
    \begin{equation}\label{ineq.G_0}
        G_0(y)\leq \frac{1}{2\pi}\ln \frac{\sqrt{2}(2t+1)}{N},\quad \forall y\in Q_{(2t+1)N^k}\setminus Q_{(2t-1)N^k},\ \forall t=0,\cdots,\frac{N-1}{2}.
    \end{equation}

     Fix a positive integer $s=[\frac{\sqrt{2}}{4}N-6]$. By the assumption on $f$ and \eqref{ineq.G_0}, we get
     \begin{align*}
         w(0)-\max_{\partial Q_{N^{k+1}}}w
         &\leq \int_{Q_{N^{k+1}}} G_0(y)f(y)dy\\
         &\leq \sum_{t=1}^{s}\int_{Q_{(2t+1)N^k}\setminus Q_{(2t-1)N^k}} G_0(y) f(y) dy\\
         &\leq \sum_{t=1}^s\frac{2}{\pi}\left(
         t\ln\frac{\sqrt{2}(2t+1)}{N}
         \right)
         \sigma^k N^{2k}.
     \end{align*}
    A simple calculation tells that
    \begin{align*}
        \sum_{t=1}^s\left(
        t\ln\frac{\sqrt{2}(2t+1)}{N}
        \right)
        \leq -\frac{1}{8}(s+1)^2+\text{lower order terms}.
    \end{align*}
    Plug this into the above estimate and the proof is finished.
\end{proof}

\subsection{Proof of Theorem~\ref{thm(main). density estimate}}
In this section we prove Theorem~\ref{thm(main). density estimate}. For clarity, we restate it here using the notations $\delta_{u}(R)$ defined at the beginning of Section~\ref{Section. essential tools}.

\begin{theorem}\label{thm. density estimate}
    Let $\lambda\leq K\leq \Lambda$ be a measurable function where $\lambda>0$, and let $u\leq M$ be a global solution of \eqref{eq. EqL variable curvature}. Then there exist constants $\epsilon,C>0$ depending on $(M,\lambda,\Lambda)$, such that for any $R\geq1$,
    \begin{equation}
        \delta_{u}(R)\leq C R^{-\epsilon}.
    \end{equation}
\end{theorem}

If $u\leq M$ globally in $\mathbb{R}^{2}$, we then let
\begin{equation}\label{substitution of variable}
    v:=M-u\geq0.
\end{equation}
It turns out that $v$ satisfies the equation
\begin{equation}
    \Delta v=K e^{2M} e^{-2v}.
\end{equation}
We denote a new curvature
\begin{equation}
    \bar{K}:=K e^{2M},
\end{equation}
then $\bar{K}\in[\bar{\lambda},\bar{\Lambda}]$ with
\begin{equation}
    \bar{\lambda}=\lambda e^{2M}>0,\quad\bar{\Lambda}=\Lambda e^{2M}.
\end{equation}
We will denote the following short-hand
\begin{equation}
    \mathcal{D}(\Omega)=\mathcal{D}_{-v}(\Omega),\quad\delta(R)=\delta_{-v}(R).
\end{equation}
We have the following initial understanding of $\delta(R)$. The proof is left to the readers.
\begin{lemma}\label{lem. naive understanding of delta}
    If $v\geq0$ in $\mathbb{R}^{2}$, then $\delta(R)\leq1$ for any $R>0$. Besides, for any $N\in2\mathbb{Z}^{+}+1$,
    \begin{equation}
        \delta(NR)\leq\delta(R).
    \end{equation}
\end{lemma}

Theorem~\ref{thm. density estimate} can then be reduced to the following version.
\begin{lemma}\label{lem. simplified density estimate}
    Assume that $v:\mathbb{R}^{2}\to[0,\infty)$ satisfies the equation
    \begin{equation}
        \Delta v=\bar{K}e^{-2v},\quad0<\bar{\lambda}\leq\bar{K}\leq\bar{\Lambda}.
    \end{equation}
    Then there exists $N\in2\mathbb{Z}^{+}+1$ depending only on $(\bar{\lambda},\bar{\Lambda})$, such that for any $k\geq0$,
    \begin{equation}
        \delta(N^{k})\leq(1-\frac{1}{2N^{2}})^{k}.
    \end{equation}
\end{lemma}
Let's first see how Lemma~\ref{lem. simplified density estimate} implies Theorem~\ref{thm. density estimate}.
\begin{proof}[Proof of Theorem~\ref{thm. density estimate}]
    By defining $v=M-u$ as in \eqref{substitution of variable} and the discussions preceding Lemma~\ref{lem. simplified density estimate}, we know that $(\bar{\lambda},\bar{\Lambda})$ depend on $(M,\lambda,\Lambda)$. Let
    \begin{equation}
        \epsilon=-\log_{N}(1-\frac{1}{2N^{2}}),
    \end{equation}
    then Lemma~\ref{lem. simplified density estimate} implies $\delta(R)\leq N^{2}R^{-\epsilon}$ for any $R\geq1$. By definition of $v$, we have
    \begin{equation}
        \delta_{u}(R)\leq N^{2}e^{2M}R^{-\epsilon}=:C R^{-\epsilon}.
    \end{equation}
    One see that $C=N^{2}e^{2M}$ also eventually depends on $(M,\lambda,\Lambda)$.
\end{proof}
In the remaining part of this section, we show Lemma~\ref{lem. simplified density estimate}. For convenience, we set
\begin{equation}
    \sigma=(1-\frac{1}{2N^{2}}),
\end{equation}
then it is clear that $\sigma\geq1/e\geq2/N^{2}$. The proof is based on an inductive argument, and the existence of $N$ will be made clear at the end of the proof.
\begin{proof}[Proof of Lemma~\ref{lem. simplified density estimate}]
    First, by Lemma~\ref{lem. naive understanding of delta}, we see Lemma~\ref{lem. simplified density estimate} is automatically correct for $k=0$. Now suppose it holds for $i=0,1,\cdots,k$ that
    \begin{equation}
        \delta(N^{i})\leq\sigma^{i},
    \end{equation}
    and we wish to show that $\mathcal{D}(Q_{N^{k+1}})\leq\sigma^{k+1}$. By moving around the origin to elsewhere, we can deduce $\delta(N^{k+1})\leq\sigma^{k+1}$, and thus the induction will go on.

    Divide $Q_{N^{k+1}}=Q_{N^{k+1}}(0)$ into a $N\times N$ grid of sub-squares, which we write as
    \begin{equation}\label{eq. subdivision of large square}
        Q_{N^{k+1}}=\bigcup_{i,j=\frac{1-N}{2}}^{\frac{N-1}{2}}Q_{N^{k}}(N^{k}i,N^{k}j)=:\bigcup_{(i,j)\in\mathcal{G}}\widetilde{Q}_{ij},\quad\mathcal{G}=\Big\{\frac{1-N}{2},\cdots,\frac{N-1}{2}\Big\}^{2}.
    \end{equation}
    
    We make the following assertion. Notice that it is true only if $N$ is chosen large (but independent of $k$).
    \begin{itemize}
        \item Assertion: There exists at lease one sub-square $\widetilde{Q}_{ij}$, such that $\mathcal{D}(\widetilde{Q}_{ij})\leq\sigma^{k}/2$.
    \end{itemize}
    If the assertion is correct, say for example $\mathcal{D}(\widetilde{Q}_{i^{*}j^{*}})\leq\sigma^{k}/2$ for some $(i^{*},j^{*})\in\mathcal{G}$. Then $\mathcal{D}(Q_{N^{k+1}})\leq\sigma^{k+1}$ follows from the computation below:
    \begin{align}
        \mathcal{D}(Q_{N^{k+1}})=&\frac{1}{N^{2}}\sum_{(i,j)\in\mathcal{G}}\mathcal{D}(\widetilde{Q}_{ij})=\frac{1}{N^{2}}\Big\{\mathcal{D}(\widetilde{Q}_{i^{*}j^{*}})+\sum_{(i,j)\neq(i^{*},j^{*})}\mathcal{D}(\widetilde{Q}_{ij})\Big\}\\
        \leq&\frac{1}{N^{2}}\Big\{\frac{\sigma^{k}}{2}+(N^{2}-1)\sigma^{k}\Big\}=\sigma^{k+1}.
    \end{align}

    Now let's prove the assertion using contradiction. Suppose otherwise
    \begin{equation}\label{evenly distributed mass, to be contradicted}
        \sigma^{k}/2\leq\mathcal{D}(\widetilde{Q}_{ij})\leq\sigma^{k},\mbox{ for all }(i,j)\in\mathcal{G}.
    \end{equation}
    
    Pick an arbitrary boundary point $x\in\partial Q_{N^{k+1}}$, and let's say $x\in\widetilde{Q}_{ij}$ for some $(i,j)\in\mathcal{G}$. By \eqref{evenly distributed mass, to be contradicted}, there exists some $y\in\widetilde{Q}_{ij}$, such that $e^{-2v(y)}\geq\sigma^{k}/2$, meaning
    \begin{equation}
        v(y)\leq-\frac{1}{2}\ln{\frac{\sigma^{k}}{2}}\leq k+1.
    \end{equation}
    We apply Lemma~\ref{lem. self similar harnack} (once or twice, depending on the difference $x-y$) to $w=v/\bar{\Lambda}$ and obtain that
    \begin{equation}
        v(x)\leq C_{1}\Big(v(y)+\bar{\Lambda}N^{2k}\sigma^{k}\ln{N}\Big)
    \end{equation}
    for a universal $C_{1}$. Therefore, we have
    \begin{equation}\label{eq. contradiction 1}
        \max_{\partial Q_{N^{k+1}}}v\leq C_{1}\Big(k+1+\bar{\Lambda}N^{2k}\sigma^{k}\ln{N}\Big).
    \end{equation}

    On the other hand, we apply Lemma~\ref{lem. maximum principle} to $w=v/\bar{\lambda}$ and obtain that
    \begin{equation}\label{contradiction 2}
        v(0)\leq \max_{\partial Q_{N^{k+1}}} v-c_3\bar{\lambda}\sigma^k N^{2k+2}
    \end{equation}
    for a universal constant $c_3>0$.

    % On the other hand, let $G_{0}(y)$ be the Green function centered at $0$ and supported in $Q_{N^{k+1}}$, ie.
    % \begin{equation}
    %     \Delta G_{0}=\delta_{0}\mbox{ in }Q_{N^{k+1}},\quad G_{0}=0\mbox{ on }\partial Q_{N^{k+1}}.
    % \end{equation}
    % It then follows that for all cubes $\widetilde{Q}_{ij}$'s with \begin{equation}
    %     (i,j)\in\mathcal{G}':=\{(i,j)\neq(0,0):i\neq\pm\frac{N-1}{2},j\neq\pm\frac{N-1}{2}\}
    % \end{equation}
    % we always have (by Harnack principle) that
    % \begin{equation}
    %     \max_{x,y\in\widetilde{Q}_{ij}}\frac{G_{0}(x)}{G_{0}(y)}\leq C_{2}.
    % \end{equation}
    
    % Since $\Delta v\geq\bar{\lambda}e^{-2v}$ and that we assume \eqref{evenly distributed mass, to be contradicted}, we can let $N\geq5$ to make sure $\mathcal{G}'\neq\emptyset$, and have
    % \begin{equation}\label{contradiction 2}
    %     v(0)-\max_{\partial Q_{N^{k+1}}}v\lesssim\sum_{(i,j)\in\mathcal{G}'}\frac{\bar{\lambda}\sigma^{k}}{2}\int_{\widetilde{Q}_{ij}}G_{0}(y)dy\leq-c_{3}\bar{\lambda}\sigma^{k}N^{2k+2}
    % \end{equation}
    % for a universal $c_{3}>0$.
    
    We combine \eqref{eq. contradiction 1} with \eqref{contradiction 2}, as well as the assumption $v\geq0$, and get that there exists a universal $c_{4}>0$ such that
    \begin{equation}
        k+1+\bar{\Lambda}N^{2k}\sigma^{k}\ln{N}\geq c_{4}\bar{\lambda}\sigma^{k}N^{2k+2}.
    \end{equation}
    We can divide $N^{2k}\sigma^{k}$ on both sides, plug in $\sigma=(1-\frac{1}{2N^{2}})$, and it follows that
    \begin{equation}\label{contradiction 3}
        \frac{k+1}{(N^{2}-1/2)^{k}}+\bar{\Lambda}\ln{N}\geq c_{4}\bar{\lambda}N^{2}.
    \end{equation}
    By choosing $N\in2\mathbb{Z}^{+}+1$ sufficiently large depending only on $(\bar{\lambda},\bar{\Lambda})$, we see \eqref{contradiction 3} fails for any $k\geq0$, and thus we have proven the assertion.
\end{proof}

\section{Growth estimates of $u$}
In this section, we shall derive growth of $u$ based on the area growth estimates and then present a new proof of Theorem~\ref{thm. EGLX}. 

\subsection{Proof of Theorem~\ref{thm(main). derivative growth rate} via Schauder estimates}
In this part we prove Theorem~\ref{thm(main). derivative growth rate}. We state its more precise form in the two propositions below.
\begin{proposition}[$L^{\infty}$ and $C^{1}$ estimate]\label{prop. C^1-estimate}
    Let $\lambda\leq K\leq \Lambda$ be a measurable function where $\lambda>0$, and let $u\leq M$ be a global solution of \eqref{eq. EqL variable curvature}. Then there exists constants $\epsilon, C>0$ depending only on $(M,\lambda,\Lambda)$, such that for any $R\geq 1$,
    \begin{equation}
        \|u\|_{L^{\infty}(B_{R})}\leq CR^{2-\epsilon},\quad\|Du\|_{L^{\infty}(B_{R})}\leq CR^{1-\epsilon^{2}/2}.
    \end{equation}
\end{proposition}
\begin{proof}
    By Theorem~\ref{thm. density estimate} and its proof, we get that there exits an integer $N\in 2\mathbb{Z}^++1$ depending only on $(M,\lambda, \Lambda)$ such that for any $x\in\mathbb{R}^2$ and any $i\geq 0$,
 \begin{align*}
     \frac{1}{N^{2i}}\int_{Q_{N^i}(x)}e^{2u}\leq C (N^{-\epsilon})^i,
 \end{align*}
 here $\epsilon=-\log_N(1-\frac{1}{2N^2})$ is positive and depends only on $(M,\lambda,\Lambda)$.
 
 Lemma~\ref{lem. self similar harnack} implies that there is a universal constant $C$ (independent on $(N,\epsilon,k)$), such that for any $k\geq 0$ and any $y\in Q_{N^k}(x)$, we have
 \begin{align*}
     -u(y)\leq C(-u(x)+N^{(2-\epsilon)k}\ln N).
 \end{align*}
 It follows that
 \begin{align*}
     -u(y)\leq C(1+R^{2-\epsilon}\ln N), \quad \forall y\in B_R.
 \end{align*}
 Note that $N$ also depends only on $(M,\lambda, \Lambda)$. We eventually get that
 \begin{align*}
     |u(x)|\leq CR^{2-\epsilon},\quad \forall y\in B_R
 \end{align*}
 for a constant C depending only on $(M,\lambda,\Lambda)$.

 Apply the $C^{1,1-\epsilon}$ estimate (see for instance, \cite[Theorem 4.16]{GT}), we get
    \begin{equation}
        R\|Du\|_{L^{\infty}(B_{R/2})}+R^{2-\epsilon}[Du]_{C^{\alpha}(B_{R/2})}\leq C\left(
        \|u\|_{L^{\infty}(B_{R})}+R^{2-\epsilon}\|Ke^{2u}\|_{L^{2/\epsilon}(B_R)}
        \right)
        \leq C R^{2-\epsilon^{2}/2},
    \end{equation}
    where $\|Ke^{2u}\|_{L^{2/\epsilon}(B_R)}$ is obtained from an interpolation between $L^{1}$ space and $L^{\infty}$ space. Hence the $C^1$-estimate follows.
\end{proof}

\begin{proposition}[$C^{2,\alpha}$ estimate]\label{prop. C^2-estimate}
    Fix an $\alpha\in (0,1)$. Let $\lambda\leq K\leq \Lambda$ be a H\"older continuous function with bounded global semi-norm $[K]_{C^{\alpha}(\mathbb{R}^{2})}$, and let $u\leq M$ be a global solution of \eqref{eq. EqL variable curvature}. 
    Then there exists constants $\epsilon, C>0$ depending only on $(M,\lambda, \Lambda)$ and $(M,\lambda,\Lambda,\alpha)$ respectively, such that for any $R\geq 1$ and $x\in B_R$, we have
    \begin{equation}
    \|D^{2}u\|_{L^{\infty}(B_R)}\leq C R^{\alpha}\ln{R},\quad[D^{2}u]_{C^{\alpha}(B_R)}\leq C\ln{R}.
    \end{equation}
\end{proposition}
\begin{proof}
    Using the assumption that $u\leq M$ and Proposition~\ref{prop. C^1-estimate}, we could estimate the $C^{\alpha}$ semi-norm of $e^{2u}$ as follows.
    
    We can apply the standard Harnack inequality to $B_{4}(x)$ for $x\in B_{R}$ (or set $k=1$ in Lemma~\ref{lem. self similar harnack}), and obtain that
    \begin{equation}
        \mbox{either }\min_{B_{2}(x)}u\geq-C_{1}\ln{R},\quad\mbox{or }\max_{B_{2}(x)}u\leq-10\ln{R}.
    \end{equation}
    In the former case, we apply the $C^{1,\alpha}$ estimate in $B_{2}(x)$ and obtain
    \begin{equation}
        [e^{2u}]_{C^{\alpha}(B_{1}(x))}\leq e^{2M}[u]_{C^{\alpha}(B_{1}(x))}\leq C_{2}\ln{R}.
    \end{equation}
    In the later case, we have
    \begin{equation}
        [e^{2u}]_{C^{\alpha}(B_{1}(x))}\leq R^{-10}[u]_{C^{\alpha}(B_{1}(x))}\leq C_{2}\ln{R},
    \end{equation}
    where we have used Proposition~\ref{prop. C^1-estimate} to control $[u]_{C^{\alpha}(B_{1}(x))}$.
    
Since $e^{2u}$ is globally bounded, the $C^{\alpha}$ semi-norm in each $B_{1}$ implies that
    \begin{equation}
        [K e^{2u}]_{C^{\alpha}(B_{R})}\leq e^{2M}[K]_{C^{\alpha}(B_{R})}+\Lambda\{e^{2M}+\sup_{x\in B_{R}}[e^{2u}]_{C^{\alpha}(B_{1}(x))}\}\leq C_{3}\ln{R}.
    \end{equation} 
Now we apply the $C^{2,\alpha}$-estimate and obtain
    \begin{align}
    &R^{2}\|D^{2}u\|_{L^{\infty}(B_{R/2})}+R^{2+\alpha}[D^{2}u]_{C^{\alpha}(B_{R/2})}\\
    \leq&C\left(
    \|u\|_{L^{\infty}(B_{R})}+R^2\|Ke^{2u}\|_{L^{\infty}(B_R)}+R^{2+\alpha}[Ke^{2u}]_{C^{\alpha}(B_R)}
     \right)
     \leq C R^{2+\alpha}\ln{R},
    \end{align}
and we have reached the desired estimates.
\end{proof}

\subsection{A new proof of Theorem~\ref{thm. EGLX}}
In this subsection, we show Theorem~\ref{thm. EGLX} using a method different from \cite{EGLX}. We start by recalling the following lemma obtained in Hang-Wang \cite{HW}.
\begin{lemma}[Hang-Wang]\label{lem. Hang-Wang}
    Let $u$ be a solution of \eqref{eq. EqL}, then the function $u_{zz}-u_{z}^{2}$ is holomorphic in the complex coordinate $\mathbb{C}\approx\mathbb{R}^{2}$.
\end{lemma}
\begin{proof}
    One can directly compute $\partial_{\bar{z}}(u_{zz}-u_{z}^{2})$. 
\end{proof}
By Theorem~\ref{thm(main). derivative growth rate}, we see $u_{zz}$ and $u_{z}$ are both of the growth rate $o(|z|)$, so $u_{zz}-u_{z}^{2}=o(|z|^{2})$. By the Liouville theorem for holomorphic functions, we conclude that
\begin{equation}
    u_{zz}-u_{z}^{2}=a z+b
\end{equation}
for some constants $a,b\in\mathbb{C}$. We claim that $a=0$ in the lemma below.
\begin{lemma}\label{lem. not linear}
    If $u_{zz}=o(|z|)$ and $u_{zz}-u_{z}^{2}=a z+b$, then $a=0$.
\end{lemma}
\begin{proof}
Suppose $a\neq 0$, then look at the image of $|z|=R\gg1$ under $u_{zz}-az-b$, which we denote by $C_R$. Since $u_{zz}=o(|z|)$, one infers that the winding number of $C_R$ about origin is $1$. 
    However, $C_R$, as the image of $|z|=R$ under $u_z^2$, its winding number must be an even number, a contradiction. Hence $a=0$.
\end{proof}
\begin{remark}
    By Nevanlinna theory, there exists solutions $u$ (without an upper bound) to \eqref{eq. EqL} such that $u_{zz}-u_z^2=az+b$ for some nonzero constant $a$. In fact, such solution $u$ could be expressed in terms of the Airy function.
\end{remark}

\begin{proof}[Proof of Theorem \ref{thm. EGLX}]
The above arguments imply that $u_{zz}-u_z^2\equiv C_0$ for a constant $C_0\in \mathbb{C}$.

 If $C_0=0$, we define $v:=e^{-u}$. Then the equation $u_{zz}-u_z^2\equiv 0$ could be converted to $\partial_{x_1x_1}v=\partial_{x_2x_2}v$ and $\partial_{x_1x_2}v=0$. Hence we could deduce that
 \begin{align*}
     D^3v=0,
 \end{align*}
 and $v$ is a quadratic polynomial. Plug $v=e^{-u}$ into the Liouville equation \eqref{eq. EqL} and we get
 \begin{align*}
     v(x)=\frac{\lambda}{2}|x-x_0|^2+\frac{1}{2\lambda},\quad \text{for some } \lambda>0.
 \end{align*}
 This shows that $u(x)=\ln (\frac{2}{1+|x|^2})$ up to a normalization.
 
\vspace{1em}
 If $C_0\neq 0$, then up to a normalization, we could assume that $C_0=-\frac{1}{4}$. Define $v:=e^{-u}$. Then the equation $u_{zz}-u_z^2\equiv -\frac{1}{4}$ could be converted to 
 \begin{align}
         \partial_{x_1x_1}v-\partial_{x_2x_2}v=v,  \label{eq.3.1} \\
         \partial_{x_1x_2}v=0.\label{eq.3.2}
 \end{align}
 From \eqref{eq.3.2} we know that $v(x_1,x_2)=F(x_1)+G(x_2)$ for some 1-variable function $F$ and $G$. Plug it into \eqref{eq.3.1} and we get
 \begin{align*}
     F''(x_1)-F(x_1)=G''(x_2)+G(x_2).
 \end{align*}
 A simple ODE calculation gives that
 \begin{align*}
     v(x_1,x_2)=c_1e^{x_1}+c_2e^{-x_1}+c_3\sin x_2+c_4\cos x_2,
 \end{align*}
 for some constants $c_1,c_2,c_3,c_4\in\mathbb{R}$.

 Plug $v=e^{-u}$ into the Liouville equation \eqref{eq. EqL} and we could deduce that $4c_1c_2=1+c_3^2+c_4^2$. It follows that
 \begin{align*}
     v(x_1,x_2)=\frac{1}{2}e^{x_1+\ln(2c_1)}+\frac{1+t^2}{2}e^{-(x+\ln(2c_1))}+t\cos(x_2+\theta),
 \end{align*}
 where $t:=\sqrt{c_3^2+c_4^2}\geq 0$. This shows that $u(x_1, x_2)=\ln \left(\frac{2 e^{x_1}}{1+t^2+2 t e^{x_1} \cos x_2+e^{2 x_1}}\right)$.
\end{proof}

\section{The Liouville equation in the half plane}
In this section, we will prove Theorem~\ref{thm. half plane}. With boundary terms taken care of properly, we shall carry out similar computations as in the full plane case, so the presentation will be a little sketchy.

Assume that $u\leq M$ in the whole $\mathbb{R}^{2}_{+}$, then we set $v=M-u\geq0$. It satisfies
\begin{equation}\label{eq. half plane: eql}
    \Delta v=K e^{2M}e^{-2v}=:\bar{K}e^{-2v},\quad \bar{\lambda}\leq\bar{K}\leq\bar{\Lambda}.
\end{equation}
The boundary condition in \eqref{eq. half plane: boundary condition} implies
\begin{equation}\label{eq. half plane: boundary inequality}
    \Big|\frac{\partial v}{\partial x_{2}}\Big|(x_{1},0)\leq|\kappa e^{M}|=:\bar{\kappa}.
\end{equation}

\subsection{Area growth estimate in the half plane}
We first derive an area growth estimate in the half plane setting. Similar to the situation for the whole space, we make the following notations:
\begin{itemize}
    \item $\mathbb{R}^{2}_{+}$: the half plane defined as $\mathbb{R}\times\mathbb{R}_{+}$.
    \item $\partial\mathbb{R}^{2}_{+}$: the $x$-axis $\mathbb{R}\times\{0\}$.
    \item $B_{R}(x)^{+}$: $B_{R}(x)^{+}=B_{R}(x)\cap\mathbb{R}^{2}_{+}$.
    \item $B_{R}^{+}$: $B_{R}^{+}=B_{R}(0)^{+}$.
    \item We denote $Q_{R}(x)$ for $x=(x_{1},x_{2})$ again by
    \begin{equation}
        Q_{R}(x):=x+[-\frac{R}{2},\frac{R}{2}]^{2},
    \end{equation}
    but this time, we require $x_{2}\geq\frac{R}{2}$ so that $Q_{R}(x)\subseteq\overline{\mathbb{R}^{2}_{+}}$;
    \item In a region $\Omega\subseteq\overline{\mathbb{R}^{2}_{+}}$, we denote the density by
    \begin{equation}
        \mathcal{D}(\Omega):=\frac{1}{|\Omega|}\int_{\Omega}e^{-2v}dx;
    \end{equation}
    \item We define the global density with scale $R$ as
    \begin{equation}
        \delta(R)=\sup_{Q_{R}(x)\subseteq\overline{\mathbb{R}^{2}_{+}}}\mathcal{D}(Q_{R}(x)).
    \end{equation}
\end{itemize}
We would like to show the following area density estimate.
\begin{proposition}\label{prop. half plane: density estimate}
    Assume that $v:\mathbb{R}^{2}_{+}\to[0,\infty)$ satisfies \eqref{eq. half plane: eql}-\eqref{eq. half plane: boundary inequality}
    Then there exists $N\in2\mathbb{Z}^{+}+1$ depending only on $(\bar{K},\bar{\kappa})$, such that for any $k\geq0$,
    \begin{equation}
        \delta(N^{k})\leq(1-\frac{1}{2N^{2}})^{k}.
    \end{equation}
\end{proposition}

Like in Section~\ref{Section. essential tools}, we choose the same integration kernel $\gamma_{R}(x)$ as in \eqref{kernel for MVP} (up to a translation). Using integration by parts, we see for any $z\in\partial\mathbb{R}^{2}_{+}$,
\begin{equation}\label{eq. MVP center origin}
   \frac{1}{2}\Big|\mathop{avg}\limits_{B_{R}(z)^{+}}w-w(z)\Big|\leq\int_{B_{R}(z)^{+}}\gamma_{R}|\Delta w|+\frac{2R}{3\pi}\cdot\max_{B_{R}(z)\cap\partial\mathbb{R}^{2}_{+}}\Big|\frac{\partial w}{\partial x_{2}}\Big|.
\end{equation}
If we assume $w\geq0$ and let $x,y\in\mathbb{R}^{2}_{+}$, then we have
\begin{align}
    &\frac{1}{2}\mathop{avg}\limits_{B_{R}(x)^{+}}w-w(x)\leq\int_{B_{R}(x)^{+}}\gamma_{R}|\Delta w|+\frac{2R}{3\pi}\cdot\max_{B_{R}(x)\cap\partial\mathbb{R}^{2}_{+}}\Big|\frac{\partial w}{\partial x_{2}}\Big|,\label{eq. MVP center high 1}\\
    &w(y)-\mathop{avg}\limits_{B_{R}(y)^{+}}w\leq\int_{B_{R}(y)^{+}}\gamma_{R}|\Delta w|+\frac{2R}{3\pi}\cdot\max_{B_{R}(y)\cap\partial\mathbb{R}^{2}_{+}}\Big|\frac{\partial w}{\partial x_{2}}\Big|+\max_{B_{R}(y)\cap\partial\mathbb{R}^{2}_{+}}w.\label{eq. MVP center high 2}
\end{align}

Next, we present the following half-plane version of Lemma~\ref{lem. self similar harnack}:
\begin{lemma}\label{lem. half space: self similar harnack}
    Let $N\in2\mathbb{Z}^{+}+1$ and $\sigma\in(2/N^{2},1)$. Let $k\geq0$ be an arbitrary integer. Assume that $w:\mathbb{R}^{2}_{+}\to[0,\infty)$ satisfies the boundary condition
    \begin{equation}
        \Big|\frac{\partial w}{\partial x_{2}}\Big|\leq1\mbox{ on }\partial\mathbb{R}^{2}_{+}
    \end{equation}
    and satisfies the equation $\Delta w=f$ such that:
    \begin{itemize}
        \item[(1)] $|f|\leq1$ everywhere in $\mathbb{R}^{2}$;
        \item[(2)] For any $Q_{N^{i}}(x)\subseteq\overline{\mathbb{R}^{2}_{+}}$ and any $i\in\{0,1,\cdots,k\}$,
        \begin{equation}
            \frac{1}{N^{2i}}\int_{Q_{N^{i}}(x)}|f|dx\leq\sigma^{i}.
        \end{equation}
    \end{itemize}
    Then there exists a universal constant $C$ (independent of $(N,\sigma,k)$), such that for any $x,y\in[-N^{k}/2,N^{k}/2]\times[0,3N^{k}]$, we have
    \begin{equation}\label{eq. half plane: Harnack formula}
        w(y)\leq C\Big(w(x)+N^{2k}\sigma^{k}\ln{N}\Big).
    \end{equation}
\end{lemma}
\begin{proof}
    The key is to first estimate $w(0)$ in Step 1-3 and then estimate $w(y)$ in Step 4-6.
    \begin{itemize}
        \item Step 1: We set $R=5N^{k}$ in \eqref{eq. MVP center high 1} and have
        \begin{equation}
            \mathop{avg}\limits_{B_{5N^{k}}(x)^{+}}w\leq 2w(x)+C_{1}N^{2k}\sigma^{k}\ln{N}+\frac{20N^{k}}{3\pi}\leq C_{2}\Big(w(x)+N^{2k}\sigma^{k}\ln{N}\Big).
        \end{equation}
        \item Step 2: We notice that definitely $B_{N^{k}}^{+}\subseteq B_{5N^{k}}(x)^{+}$, so
        \begin{equation}
            \mathop{avg}\limits_{B_{N^{k}}^{+}}w\leq C_{3}\Big(w(x)+N^{2k}\sigma^{k}\ln{N}\Big).
        \end{equation}
        \item Step 3: We apply \eqref{eq. MVP center origin} with $R=N^{k}$ and obtain an upper bound of $w(0)$:
        \begin{equation}
            w(0)\leq C_{4}\Big(w(x)+N^{2k}\sigma^{k}\ln{N}\Big).
        \end{equation}
        \item Step 4: We apply \eqref{eq. MVP center origin} with $R=5N^{k}$ and obtain
        \begin{equation}
            \mathop{avg}\limits_{B_{5N^{k}}^{+}}w\leq C_{5}\Big(w(x)+N^{2k}\sigma^{k}\ln{N}\Big).
        \end{equation}
        \item Step 5: As $B_{N^{k}}(y)^{+}\subseteq B_{5N^{k}}^{+}$, we have
        \begin{equation}
            \mathop{avg}\limits_{B_{N^{k}}(y)^{+}}w\leq C_{6}\Big(w(x)+N^{2k}\sigma^{k}\ln{N}\Big).
        \end{equation}
        \item Step 6: Finally, we apply \eqref{eq. MVP center high 2} with $R=N^{k}$ and obtain
        \begin{equation}
            w(y)\leq C_{7}\Big(w(x)+N^{2k}\sigma^{k}\ln{N}\Big).
        \end{equation}
    \end{itemize}
    The details are omitted as we just repeat the computations in the proof of Lemma~\ref{lem. self similar harnack}.
\end{proof}

\begin{proof}[Proof of Proposition~\ref{prop. half plane: density estimate}]
    We only need to prove \eqref{eq. contradiction 1} in the half-plane setting, while all other steps in the proof of Lemma~\ref{lem. simplified density estimate} can be directly copied here.

    Let $\widetilde{Q}_{ij}$ be a a sub-square defined in \eqref{eq. subdivision of large square}. We say $\widetilde{Q}_{ij}$ is "high", if
    \begin{equation}
        \widetilde{Q}_{ij}\subseteq\{x_{2}\geq\frac{3}{2}N^{k}\}.
    \end{equation}
    We say $\widetilde{Q}_{ij}$ is "low", if
    \begin{equation}
        \widetilde{Q}_{ij}\subseteq\{x_{2}\leq3N^{k}\}.
    \end{equation}
    
    Clearly any sub-square $\widetilde{Q}_{ij}$ is either "high" or "low" (or both), so \eqref{eq. contradiction 1} is obtained from applying Lemma~\ref{lem. self similar harnack} to "high" sub-squares (see Remark~\ref{rmk. self similar harnack domain}), and applying Lemma~\ref{lem. half space: self similar harnack} to "low" sub-squares.
\end{proof}
\subsection{Schauder estimates in the half plane}
In this subsection, we prove the half-plane version of Schauder estimates.

We first provide an ABP estimate in the half plane.
\begin{lemma}\label{lem. ABP}
    Assume that $w:\overline{B_{R}^{+}}\to\mathbb{R}$ vanishes at $\partial B_{R}\cap\mathbb{R}^{2}_{+}$, then
    \begin{equation}
        \|w\|_{L^{\infty}(\overline{B_{R}^{+}})}\leq CR\Big(\|\Delta w\|_{L^{2}(B_{R}^{+})}+\|\frac{\partial w}{\partial x_{2}}\|_{L^{\infty}(B_{R}\cap\partial\mathbb{R}^{2}_{+})}\Big).
    \end{equation}
\end{lemma}
\begin{proof}
    For simplicity let's assume $R=1$ and
    \begin{equation}
        \|\Delta w\|_{L^{2}(B_{1}^{+})}+\|\frac{\partial w}{\partial x_{2}}\|_{L^{\infty}(B_{1}\cap\partial\mathbb{R}^{2}_{+})}\leq1.
    \end{equation}
    
    We argue by contradiction. Pick an arbitrary point $z\in\overline{B_{1}^{+}}$, and suppose that $w(z)\leq-H$ (or $\geq H$) for some large $H$.
    Let $\vec{v}=(v_{1},v_{2})$ be a vector satisfying
    \begin{equation}
        |\vec{v}|<H/2,\quad v_{2}>1.
    \end{equation}
    We start from $h=-\infty$ and increase $h$ till the linear function
    \begin{equation}
        L_{\vec{v},h}(x)=\vec{v}\cdot x+h
    \end{equation}
    touches the graph of $w$ for the first time. Let $y\in\overline{B_{1}^{+}}$ be a contact point, so that
    \begin{equation}
        w(x)\geq L_{\vec{v},h}(x)\mbox{ in }\overline{B_{1}^{+}},\quad w(y)=L_{\vec{v},h}(y).
    \end{equation}

    It's not hard to verify the following facts:
    \begin{itemize}
        \item $y$ is an interior point of $B_{1}^{+}$;
        \item $D^{2}w(y)$ is positive-semi-definite and $\det{D^{2}w(y)}\leq(\Delta w)^{2}/4$.
    \end{itemize}
    Therefore, we reach a contradiction if $H$ is too large, since otherwise
    \begin{equation}
        1\geq\|\Delta w\|_{L^{2}(B_{1}^{+})}^{2}\geq4\Big|\{\vec{v}:|\vec{v}|<H/2,\ v_{2}>1\}\Big|\geq\frac{\pi}{3}H^{2}.
    \end{equation}
\end{proof}

We next present a $C^{\alpha}$ Schauder estimate.
\begin{lemma}\label{lem. half space: Schauder 0}
    Assume that $w:B_{R}^{+}\to\mathbb{R}$ satisfies the equation
    \begin{equation}
        \Delta w=f,\quad\frac{\partial w}{\partial x_{2}}\Big|_{B_{R}\cap\partial\mathbb{R}^{2}_{+}}=g.
    \end{equation}
    Then for any $\alpha\in(0,1)$, there exists a constant $C=C(\alpha)$, such that
    \begin{equation}
        R^{\alpha}[w]_{C^{\alpha}(B_{R/2}^{+})}\leq C\Big(R\|f\|_{L^{2}(B_{R}^{+})}+R\|g\|_{L^{\infty}(B_{R}\cap\partial\mathbb{R}^{2}_{+})}+\|w\|_{L^{\infty}(B_{R}^{+})}\Big).
    \end{equation}
\end{lemma}
\begin{proof}
    For simplicity, we assume $R=1$ and
    \begin{equation}
        \|f\|_{L^{2}(B_{1}^{+})}+\|g\|_{L^{\infty}(B_{1}\cap\partial\mathbb{R}^{2}_{+})}+\|w\|_{L^{\infty}(B_{1}^{+})}\leq1.
    \end{equation}
    We let $\rho\in(0,1)$ yet to be decided and denote
    \begin{equation}
        A_{k}=\mathop{osc}\limits_{B_{r_{k}}^{+}}w(x),\quad r_{k}=\rho^{k}.
    \end{equation}
    Clearly, we have $A_{0}\leq\|w\|_{L^{\infty}(B_{1}^{+})}\leq1$.

    Let $\widetilde{w_{k}}$ be a harmonic replacement of $w$ in $B_{r_{k}}^{+}$ such that
    \begin{equation}
        \Delta\widetilde{w_{k}}=0\mbox{ in }B_{r_{k}}^{+},\quad\widetilde{w_{k}}\Big|_{\partial B_{r_{k}}\cap\mathbb{R}^{2}_{+}}=w,\quad\frac{\partial\widetilde{w_{k}}}{\partial x_{2}}\Big|_{B_{r_{k}}\cap\partial\mathbb{R}^{2}}=0.
    \end{equation}
    By the standard Schauder estimate for harmonic functions we have
    \begin{equation}
        \mathop{osc}\limits_{B_{r_{k+1}}^{+}}\widetilde{w}\leq C\rho A_{k}.
    \end{equation}
    By Lemma~\ref{lem. ABP}, we have
    \begin{equation}
        \mathop{osc}\limits_{B_{r_{k+1}}^{+}}(w-\widetilde{w})\leq C\rho^{k}.
    \end{equation}
    These estimates imply that $A_{k+1}\leq C\rho A_{k}+C\rho^{k}$. By choosing $\rho$ sufficiently small, we conclude that $A_{k}\leq C\rho^{k\alpha}$ for all $k\geq0$.
\end{proof}
Similarly we have the following $C^{1,\alpha}$ and $C^{2,\alpha}$ Schauder estimate. In fact, the 
$C^{2,\alpha}$ estimate can be found in Gilbarg-Trudinger \cite{GT} (Theorem 6.26). The proof will be omitted.
\begin{lemma}\label{lem. half space: Schauder 12}
    Assume that $w:B_{R}^{+}\to\mathbb{R}$ satisfies the equation
    \begin{equation}
        \Delta w=f,\quad\frac{\partial w}{\partial x_{2}}\Big|_{B_{R}\cap\partial\mathbb{R}^{2}_{+}}=g.
    \end{equation}
    \begin{itemize}
        \item[(1)] If $f\in L^{p}(B_{R}^{+})$ for $p=\frac{2}{1-\alpha}$ and $g\in C^{\alpha}(B_{R}\cap\partial\mathbb{R}^{2}_{+})$, then
        \begin{align}
            &R[w]_{C^{0,1}(B_{R/2}^{+})}+R^{1+\alpha}[w]_{C^{1,\alpha}(B_{R/2}^{+})}\\
            \leq&C\Big(R^{1+\alpha}\|f\|_{L^{p}(B_{R}^{+})}+R\|g\|_{L^{\infty}(B_{R}\cap\partial\mathbb{R}^{2}_{+})}+R^{1+\alpha}[g]_{C^{\alpha}(B_{R}\cap\partial\mathbb{R}^{2}_{+})}+\|w\|_{L^{\infty}(B_{R}^{+})}\Big);
        \end{align}
        \item[(2)] If $f\in C^{\alpha}(B_{R}^{+})$ and $g\in C^{1,\alpha}(B_{R}\cap\partial\mathbb{R}^{2}_{+})$, then
        \begin{align}
            R^{2}[w]_{C^{1,1}(B_{R/2}^{+})}&+R^{2+\alpha}[w]_{C^{2,\alpha}(B_{R/2}^{+})}
            \leq C\Big(R^{2}\|f\|_{L^{\infty}(B_{R}^{+})}+R^{2+\alpha}[f]_{C^{\alpha}(B_{R}^{+})}\\
            &+R\|g\|_{L^{\infty}(B_{R}\cap\partial\mathbb{R}^{2}_{+})}+R^{2+\alpha}[g]_{C^{1,\alpha}(B_{R}\cap\partial\mathbb{R}^{2}_{+})}+\|w\|_{L^{\infty}(B_{R}^{+})}\Big).
        \end{align}
    \end{itemize}
\end{lemma}

We now state the estimates of $u$ up to second derivatives similar to Theorem~\ref{thm(main). derivative growth rate}.
\begin{proposition}
    Assume that $0<\lambda\leq K\leq\Lambda$, $|\kappa|\leq\Lambda$ and
    \begin{equation}
        [K]_{C^{\alpha}(\mathbb{R}^{2}_{+})}+[\kappa]_{C^{1,\alpha}(\partial\mathbb{R}^{2}_{+})}\leq\Lambda.
    \end{equation}
    If $u\leq M$ satisfies \eqref{eq. half plane: boundary condition}, then there exists $\epsilon$ depending only on $(M,\lambda,\Lambda)$ such that
    \begin{equation}
        \|u\|_{L^{\infty}(B_{R}^{+})}=O(R^{2-\epsilon}),\quad\|Du\|_{L^{\infty}(B_{R}^{+})}=O(R^{1-\frac{1-\alpha}{2}\epsilon}),\quad\|D^{2}u\|_{L^{\infty}(B_{R}^{+})}=O\Big(R^{\alpha}(\ln{R})^{2}\Big).
    \end{equation}
\end{proposition}
\begin{proof}
    The idea is to estimate $\|u\|_{C^{\alpha}(B_{R}^{+})}$ and then improve the regularity. Throughout the proof, $C_{i}$'s refer to constants depending only on $(M,\lambda,\Lambda,\alpha)$.
    \begin{itemize}
        \item Step 1: By Proposition~\ref{prop. half plane: density estimate} and Lemma~\ref{lem. half space: self similar harnack}, we see
        \begin{equation}
            -u(x)\leq C_{1}(1+|x|^{2-\epsilon}).
        \end{equation}
        \item Step 2: By Proposition~\ref{prop. half plane: density estimate}, we see
        \begin{equation}\label{eq. half plane: laplace Lp}
            \|\Delta u\|_{L^{p}(B_{R}^{+})}\leq C_{2}(p)\cdot(1+R^{(2-\epsilon)/p}),
        \end{equation}
        so we choose $p=2$ and obtain from Lemma~\ref{lem. half space: Schauder 0} that
        \begin{equation}
            [u]_{C^{\alpha}(B_{R}^{+})}\leq C_{3}(1+R^{2-\alpha-\epsilon/2}).
        \end{equation}
        \item Step 3: We now show for $R\geq100$,
        \begin{equation}\label{eq. half space: boundary C alpha}
            \Big[\frac{\partial u}{\partial x_{2}}\Big]_{C^{\alpha}(B_{R}\cap\partial\mathbb{R}^{2}_{+})}\leq C_{4}\ln{R}.
        \end{equation}
        The idea is similar to that in Proposition~\ref{prop. C^2-estimate}. In fact, as $\frac{\partial u}{\partial x_{2}}=-\kappa e^{u}$ is globally bounded, it suffices to show
        \begin{equation}
            [e^{u}]_{C^{\alpha}(B_{1}(x)^{+})}\leq C_{4}\ln{R}.
        \end{equation}
        for any $x\in B_{R}\cap\partial\mathbb{R}^{2}_{+}$. By choosing $k=1$ in Lemma~\ref{lem. half space: self similar harnack} we see that
        \begin{equation}
            \mbox{either }\min_{B_{2}(x)^{+}}u\geq-C_{5}\ln{R},\quad\mbox{or }\max_{B_{2}(x)^{+}}u\leq-10\ln{R}.
        \end{equation}
        In the first case, we apply Lemma~\ref{lem. half space: Schauder 0} and obtain
        \begin{equation}
            [e^{u}]_{C^{\alpha}(B_{1}(x)^{+})}\leq e^{M}[u]_{C^{\alpha}(B_{1}(x)^{+})}\leq C_{4}\ln{R}.
        \end{equation}
        In the second case, we have
        \begin{equation}
            [e^{u}]_{C^{\alpha}(B_{1}(x)^{+})}\leq R^{-10}[u]_{C^{\alpha}(B_{1}(x)^{+})}\leq C_{3}\leq C_{4}\ln{R},
        \end{equation}
        where we have used Step 2 to control $[u]_{C^{\alpha}(B_{1}(x)^{+})}$.
        \item Step 4: We apply Lemma~\ref{lem. half space: Schauder 12} (1) and \eqref{eq. half plane: laplace Lp}-\eqref{eq. half space: boundary C alpha}, and obtain
        \begin{equation}
            [u]_{C^{0,1}(B_{R}^{+})}+R^{\alpha}[u]_{C^{1,\alpha}(B_{R}^{+})}\leq C_{6}(1+R^{1-\frac{1-\alpha}{2}\epsilon}).
        \end{equation}
        \item Step 5: We show for $R\geq100$,
        \begin{equation}\label{eq. half space: boundary C 1 alpha}
            \Big[\frac{\partial u}{\partial x_{2}}\Big]_{C^{1,\alpha}(B_{R}\cap\partial\mathbb{R}^{2}_{+})}\leq C_{7}(\ln{R})^{2}.
        \end{equation}
        It suffices to estimate the $C^{1,\alpha}$ norm for $e^{u}$. For any $x,y\in B_{R}\cap\partial\mathbb{R}^{2}_{+}$ with $|x-y|\leq1$,
        \begin{equation}
            \frac{|\nabla e^{u(x)}-\nabla e^{u(y)}|}{|x-y|^{\alpha}}\leq\{[u]_{C^{\alpha}(B_{1}(x)^{+})}[u]_{C^{0,1}(B_{1}(x)^{+})}+[u]_{C^{1,\alpha}(B_{1}(x)^{+})}\}\cdot\max_{B_{1}(x)^{+}}e^{u}.
        \end{equation}
        We then follow a similar method as in Step 3 and get \eqref{eq. half space: boundary C 1 alpha}.
        \item Step 6: We argue similarly as in Proposition~\ref{prop. C^2-estimate}, first obtaining that
        \begin{equation}
            [K e^{2u}]_{C^{\alpha}(B_{R}^{+})}\leq C_{9}\ln{R},
        \end{equation}
        and then reaching the desired estimate of $\|D^{2}u\|_{L^{\infty}(B_{R}^{+})}$ using Lemma~\ref{lem. half space: Schauder 12} (2).
    \end{itemize}
\end{proof}

\subsection{Classification of half plane solutions with an upper bound}
In this subsection, we study the classification problem of \eqref{eq. half plane: boundary condition} for solutions with a global upper bound and prove Theorem~\ref{thm. half plane}.

\begin{lemma} \label{lem. half}
    Assume that $u:\mathbb{R}^{2}_{+}\to\mathbb{R}$ satisfies \eqref{eq. half plane: boundary condition}, then $u_{zz}-u_{z}^{2}$ is holomorphic and takes real value on $x_2=0$. 
\end{lemma}
\begin{proof}
    By taking the tangential derivative of the boundary condition in \eqref{eq. half plane: boundary condition}, we see
    \begin{equation}
        0=\frac{\partial}{\partial x_{1}}(\frac{\partial u}{\partial x_{2}}+\kappa e^{u})=u_{x_{1}x_{2}}+\kappa e^{u}u_{x_{1}}=u_{x_{1}x_{2}}-u_{x_{2}}u_{x_{1}}, \quad \text{on $x_2=0$}.
    \end{equation}
 Therefore, the imaginary part
    \begin{equation}
        Im(u_{zz}-u_{z}^{2})=\frac{1}{2}(u_{x_{1}}u_{x_{2}}-u_{x_{1}x_{2}})=0.
    \end{equation}
\end{proof}

\begin{proof}[Proof of Theorem \ref{thm. half plane}]
In view of Lemma \ref{lem. half}, we may use Schwartz reflection to extend $u_{zz}-u_z^2$ to an entire holomorphic function $f$. Since $Du$ and $D^{2}u$ are of order $o(|z|)$, it follows that $f$ is of order $o(|z|^2)$, thus 
\begin{equation}
u_{zz}-u_z^2=az+b,\quad a,b\in\mathbb{R}.
\end{equation}

Using a similar method as in Lemma \ref{lem. not linear}, we have $a=0$ and thus
\begin{equation}
    u_{zz}-u_{z}^{2}=C_0\in\mathbb{R}.
\end{equation}
Precisely speaking, we argue by contradiction and suppose $a\neq0$. We look at the image of the curve $z(\theta)=R e^{i\theta}, \theta\in[0,\pi]$ when $R\gg 1$. Let 
\begin{equation}
    f(\theta)=u_{zz}-az-b\Big|_{z=z(\theta)}=u_{z}^{2}\Big|_{z=z(\theta)}.
\end{equation}
\begin{itemize}
    \item On the one hand, $f(\theta)=u_{zz}-az-b\Big|_{z=z(\theta)}$ and since $|u_{zz}|=o(|z|)$, we see $\arg f(\theta)=\arg(-a)+\theta+o(1)$, so
    \begin{equation}
        \int_{0}^{\pi}\frac{d}{d\theta}\arg{f(\theta)}d\theta=\pi+o(1).
    \end{equation}
    \item On the other hand,  $f(\theta)=u_{z}^{2}\Big|_{z=z(\theta)}$. Notice that $|u_{z}|\geq c\sqrt{R}$, while by \eqref{eq. half plane: boundary condition},
    \begin{equation}
        |Im(u_{z})|=|\frac{\kappa}{2}e^{u}|\leq|\kappa|e^{M}\mbox{ when }z\in\mathbb{R}.
    \end{equation}
   Hence $u_{z}$ is "almost real" on the real line, so
    \begin{equation}
        \int_{0}^{\pi}\frac{d}{d\theta}\arg{f(\theta)}d\theta=2\int_{0}^{\pi}\frac{d}{d\theta}\arg{u_{z}}d\theta=2k\pi+o(1).
    \end{equation}
\end{itemize}
Therefore we have reached a contradiction.

 Like in the proof of Theorem~\ref{thm. EGLX}, we set $v:=e^{-u}$ and obtain from the fact $u_{zz}-u_z^2=C_0$ (where $C_0\in\mathbb{R}$ by Lemma \ref{lem. half}) that
\begin{align}
        \partial_{x_1x_1}v-\partial_{x_2x_2}v=&-4C_0 v, \\
        \partial_{x_1x_2}v=&0.
\end{align}
Due to the restricted normalization \[
u_{\lambda, a}(x_1, x_2):=u(\lambda x_1+a, \lambda x_2)+\ln \lambda \quad (\lambda>0 , a\in \mathbb{R}),
\]  we shall have three cases:
\begin{enumerate}
    \item $u_{zz}-u_z^2\equiv 0\implies $
    \[
    v=\frac{1+x_1^2+(x_2+b)^2}{2};
    \]
   \item $u_{zz}-u_z^2\equiv -\frac{1}{4}\implies $
   \[
   v=\frac{1}{2}e^{x_1}+\frac{1+t^2}{2}e^{-x_1}+t\cos (x_2+b);
   \]
   \item $u_{zz}-u_z^2\equiv \frac{1}{4} \implies $
   \[
   v=\frac{1}{2}e^{x_2+b}+\frac{1+t^2}{2}e^{-(x_2+b)}+t\cos(x_1).
   \]
\end{enumerate}
Taking into account of the boundary condition, we get 
\[
\frac{\partial v}{\partial x_2}=-e^{-u} \frac{\partial u}{\partial x_2}=\kappa, \quad \text{on $x_2=0$.}
\]
Then the desired conclusion follows by simple computations.

\end{proof}

\end{document}